\documentclass[11pt,twoside,a4paper]{amsart}
\usepackage[french]{babel}
\usepackage{amssymb,bbm}
\usepackage{amsmath}
\usepackage{amsthm}
\usepackage{hyperref}
\usepackage[dvips]{graphics}
\usepackage{graphicx}
\usepackage{latexsym}
\usepackage[normal]{subfigure}
\input{epsf.sty}
\include{psfig}
\usepackage[all]{xy}
\usepackage{eufrak}
\usepackage{mathrsfs}
\usepackage{lipsum}
\usepackage{mdframed}
\usepackage{perpage}
\usepackage{amsfonts}
\usepackage{url,color,srcltx}
\begin{document}
\def\K{\mathbb{K}}
\def\R{\mathbb{R}}
\def\C{\mathbb{C}}
\def\Z{\mathbb{Z}}
\def\Q{\mathbb{Q}}
\def\D{\mathbb{D}}
\def\N{\mathbb{N}}
\def\T{\mathbb{T}}
\def\P{\mathbb{P}}
\def\A{\mathscr{A}}
\def\CC{\mathscr{C}}
\renewcommand{\theequation}{\thesection.\arabic{equation}}
\newtheorem{theorem}{Th\'eor\`eme}[section]
\newtheorem{lemma}{Lemme}[section]
\newtheorem{corollary}{Corollaire}[section]
\newtheorem{prop}{Proposition}[section]
\newtheorem{definition}{D\'efinition}[section]
\newtheorem{remark}{Remarque}[section]
\newtheorem{example}{Exemple}[section]
\newtheorem{notation}{Notation}
\newtheorem{con}{Cons\'equence}
\bibliographystyle{plain}
                                    
\title[R\'esolution du $\partial \bar{\partial}$ pour les formes ayant une valeur au bord au sens des courants~~]{\textbf {R{\'e}solution du $\partial \bar{\partial}$ pour les formes diff{\'e}rentielles ayant une valeur au bord au sens des courants d{\'e}finies dans un domaine contractile fortement pseudoconvexe d'une vari{\'e}t{\'e} complexe} }
\author[ S.\  Sambou   \& S.\  Sambou ]
{Souhaibou Sambou  \& Salomon Sambou}
\address{D{\'e}partement de Math{\'e}matiques\\UFR des Sciences et Technologies \\ Universit{\'e} Assane Seck de Ziguinchor, BP: 523 (S{\'e}n{\'e}gal)}

\email{sambousouhaibou@yahoo.fr  \& ssambou@univ-zig.sn }

\date{}
\maketitle
%%%%%%%%%%%%%%%%%%%%%%%%%%%%%%%%%%%%%%%%%%%%%%%%%%%%%%%%%%%%%%%%%%%%%%%%%%%%%%%%%%%%%%%%%%%%%%%%%%
       \begin{abstract}
On r{\'e}sout le $\partial \bar{\partial}$ pour les formes diff{\'e}rentielles admettant une valeur au bord au sens des courants d{\'e}finie sur un domaine contractile fortement pseudoconvexe d'une vari{\'e}t{\'e} complexe .
%%%%%%%%%%%%%%%%%%%%%%%%%%%%%%%%%%%%%%%%%%%%%%%%%%%%%%%%%%%%%%%%%%%%%%%%%%%%%%%%%%%%%%%%%%%%%%%%%%%%%%%%%%
\vskip 2mm
\noindent
{\normalsize A}{\tiny BSTRACT.}
We solve the $\partial \bar{\partial}$-problem for a form with distribution boundary value on a strongly pseudoconvex contractible domain of  a complex manifold.
\vskip 2mm
\noindent
\vskip 1.3mm
\noindent
\textit{Classification math{\'e}matique 2010~:}  32F32.
\end{abstract} 
%%%%%%%%%%%%%%%%%%%%%%%%%%%%%%%%%%%%%%%%%%%%%%%%%%%%%%%%%%%%%%%%
\section*{Introduction}
Soient $M$ une vari{\'e}t{\'e} diff{\'e}rentiable et $\Omega \subset M$ un domaine. Dans ce travail, on cherche {\`a} r{\'e}soudre l'{\'e}quation $\partial \bar{\partial}$ pour les formes diff{\'e}rentielles ayant une valeur au bord au sens des courants.\\
Nous suivons les m{\^e}mes d{\'e}marches de r{\'e}solution de $[6]$ o{\`u} le m{\^e}me probl{\`e}me a {\'e}t{\'e} r{\'e}solu pour des ouverts de $\mathbb{C}^n$. Ces r{\'e}sultats de $[6]$ peuvent {\^e}tre vu comme les analogues locaux de ce travail. Il est connu qu'on ne passe  par recollement des r{\'e}sultats locaux aux r{\'e}sultats globaux, la partition de l'unit{\'e} n'{\'e}tant pas holomorphe. Nous allons passer du local au global par l'utilisation classique de la th{\'e}orie des faisceaux. Le r{\'e}sultat obtenu dans cette direction est le suivant:
\begin{theorem} \label{zhou}
Soient $M$ une vari{\'e}t{\'e} analytique complexe de dimension $n$ et $\Omega \subset \subset M$ un domaine contractile compl{\'e}tement strictement pseudoconvexe {\`a} bord lisse de classe $C^\infty$. Supposons que $H^j(b\Omega)$ trivial pour $1 \leq j \leq 2n-2$. Alors l'{\'e}quation $\partial \bar{\partial} u = f$, o{\`u} $f$ est une $(p,q)$-forme diff{\'e}rentielle de classe $C^\infty$, $d$-ferm{\'e}e ayant une valeur au bord au sens des courants avec $1 \leq p \leq n$ et $1 \leq q \leq n$ admet une solution $u$ qui est une $(p-1,q-1)$-forme diff{\'e}rentielle de classe $C^\infty$ avec une valeur au bord au sens des courants.
\end{theorem}
Nous consid{\'e}rons {\'e}galement la version concave du th{\'e}or{\`e}me pr{\'e}c{\'e}dent et on obtient:
\begin{theorem} \label{aff}
Soient $M$ une vari{\'e}t{\'e} analytique complexe de dimension $n$ et $D \subset \subset M$ un domaine contractile compl{\'e}tement strictement pseudoconvexe {\`a} bord lisse de classe $C^\infty$. Supposons que $M$ est une extension $n-1$-convexe de $D$ et une extension contractile de $D$ avec $H^j(bD)$ trivial  pour $1 \leq j \leq 2n-2$. Posons $\Omega = M \setminus \bar{D}$. Si $\stackrel{\circ}{\bar{\Omega}} = \Omega$ et $f$ est une $(p,q)$-forme diff{\'e}rentielle de classe $C^\infty$, $d$-ferm{\'e}e et admettant une valeur au bord au sens des courants sur $\Omega$ avec $1 \leq p \leq n-1$ et $1 \leq q \leq n-1$, alors il existe une $(p-1,q-1)$-forme diff{\'e}rentielle $u$ de classe $C^\infty$ d{\'e}finie sur $\Omega$ ayant une valeur au bord au sens des courants telle que $\partial \bar{\partial} u = f$.
\end{theorem}
%%%%%%%
\section{Pr{\'e}liminaires et notations}
\begin{definition}{(voir $[2]$)}
Soient $X$ une vari{\'e}t{\'e} diff{\'e}rentiable de dimension $n$ et $\Omega \subset \subset X$ un domaine contractile. On dit que $X$ est une extension contractile de $\Omega$, s'il existe une suite $(\Omega_n)_n$ exhaustive de domaines contractiles telle que\\
$\forall n \in \mathbb{N}$,  $\Omega \subset \subset \Omega_n \subset \subset X$.
\end{definition}
\begin{example}
Quand $X = \mathbb{C}^n$, alors $\mathbb{C}^n$ est une extension contractile de la boule unit{\'e} $\mathcal{B}$.
\end{example}
\begin{definition}
Soit $X$ une vari{\'e}t{\'e} analytique complexe de dimension $n$.
\begin{itemize}
\item[(1)] Une fonction $\rho$ de classe $C^\infty$ sur $X$ est dite $n$-convexe (respectivement $n$-concave) si sa forme de L{\'e}vi poss{\`e}de $n$ valeurs propres strictement positives (respectivement strictement n{\'e}gatives).
\item[(2)] Soit $\Omega \subset \subset X$ un domaine relativement compact de $X$. $\Omega$ est dit compl{\'e}tement strictement pseudoconvexe s'il existe une fonction $n$-convexe $\varphi$ d{\'e}finie dans un voisinage $U_{\bar{\Omega}}$ de $\bar{\Omega}$ telle que\\ $\Omega = \{ z \in U_{\bar{\Omega}}~\vert~ \varphi(z) < 0 \}$.
\item[(3)] $X$ est une extension $(n-1)$-concave de $\Omega$ si :
\begin{itemize}
\item[(i)] $\Omega$ rencontre toutes les composantes connexes de $X$.
\item[(ii)] Il existe une fonction $n$-concave $\varphi$ d{\'e}finie sur un voisinage $U$ de $X \setminus \Omega$ telle que $\Omega \cap U = \{ z \in U~\vert~ \varphi(z) <0\}$ et pour tout r{\'e}el $\alpha$ avec $0< \alpha < \sup_{z \in U} \varphi(z)$ l'ensemble $\{z \in U~\vert~0 \leq \varphi(z) \leq \alpha\}$ est compact.
\end{itemize}
\end{itemize}
\end{definition}
\begin{definition}   
Soient $M$ une vari{\'e}t{\'e} diff{\'e}rentiable et $\Omega \subset \subset M$ un domaine {\`a} bord lisse de classe $C^\infty$ de fonction d{\'e}finissante $\rho$. Posons $\Omega_\varepsilon = \{z \in \Omega / \rho(z)<-\varepsilon\}$ et $b\Omega_\varepsilon$ le bord de $\Omega_\varepsilon$.\\
Soit $f$ une fonction de classe $C^\infty$ sur $\Omega$. On dit que $f$ admet une valeur au bord au sens des distributions, s'il existe une distribution $T$ d{\'e}finie sur le bord $b\Omega$ de $\Omega$ telle que pour toute fonction $\varphi \in C^\infty (b\Omega)$, on ait:\\
\[ \lim_{\varepsilon \rightarrow 0} \int_{b\Omega_\varepsilon} f \varphi_\varepsilon d\sigma = <T,\varphi>\]
o{\`u} $\varphi_\varepsilon = i_{\varepsilon}^* \tilde{\varphi}$  avec $\tilde{\varphi}$ une extension de $\varphi$ {\`a} $\Omega$ et $i_{\varepsilon} : b\Omega_\varepsilon  \rightarrow b\Omega$ l'injection canonique; $d\sigma$ d{\'e}signe l'{\'e}l{\'e}ment de volume.\\
Une forme diff{\'e}rentielle de classe $C^\infty$ sur $\Omega$ admet une valeur au bord au sens des courants si ses coefficients ont une valeur au bord au sens des distributions.
\end{definition}
\begin{definition}
On dit qu'une fonction $f$ de classe $C^\infty$ d{\'e}finie sur $\Omega$ est {\`a} croissance polynomiale d'ordre $N \geq 0$, s'il existe une constante $C$ telle que pour tout $z \in \Omega$, on a:\\
\begin{center}
$\vert f(z) \vert \leq \frac{C}{d(z)^N}$
\end{center}
o{\`u}  $d(z)$ d{\'e}signe la distance de $z$ au bord de $\Omega$ .
\end{definition}
\begin{definition}
Soit $\Omega \subset M$ un domaine d'un vari{\'e}t{\'e} diff{\'e}rentiable $M$. Un courant $T$ d{\'e}fini sur $\Omega$ est dit prolongeable si $T$ est la restriction {\`a} $\Omega$ d'un courant $\tilde{T}$ d{\'e}fini sur $M$.\\
Il est connu d'apr{\`e}s $[3]$, que si $\stackrel{\circ}{\bar{\Omega}} = \Omega$, alors les courants prolongeables sont les duaux topologiques des formes diff{\'e}rentielles {\`a} support compact sur $\bar{\Omega}$.\\
Nous aurons dans toute la suite {\`a} consid{\'e}rer des domaines $\Omega$ v{\'e}rifiant $\stackrel{\circ}{\bar{\Omega}} = \Omega$.
\end{definition}
%%%%%%%%%%%%%%%%%%%%%%%%%%%%%%%%%%%%%%%%%%%%%%%%%%%%%%%%%%%%%%%%%%%%%%%%%%%%%%%%%%%%%%%%%%%%%%%%%%%%%%%%%%
\begin{notation}
Soit $\Omega$ un ouvert d'une vari{\'e}t{\'e} analytique complexe $M$ de dimension $n$.\\
On note $O_\Omega$ le faisceau des fonctions holomorphes sur $\Omega$, $\check{O}_\Omega$ celui sur $\bar{\Omega}$ des germes de fonctions holomorphes sur $\Omega$ ayant une valeur au bord au sens des distributions et $\mathcal{F}^{0,r}(\Omega)$ celui sur $\bar{\Omega}$ des $(0,r)$-formes diff{\'e}rentielles sur $\Omega$ ayant une valeur au bord au sens des courants.\\
On note $\check{H}^r(\Omega)$ le $r^{i\grave{e}me}$ groupe de cohomologie de De Rham des courants prolongeables d{\'e}finis sur $\Omega$, $H^r(\Omega)$ le $r^{i\grave{e}me}$ groupe de cohomologie de De Rham des formes diff{\'e}rentiables de classe $C^\infty$ d{\'e}finies sur $\Omega$ et $H^r(b\Omega)$ le $r^{i\grave{e}me}$ groupe de cohomologie de De Rham des formes diff{\'e}rentiables de classe $C^\infty$ d{\'e}finies sur $b\Omega$. Le $r^{i\grave{e}me}$ groupe de cohomologie de De Rham des formes diff{\'e}rentielles ayant une valeur au bord au sens des courants sur  $\Omega$ est not{\'e} $\tilde{H}^r(\Omega)$.\\
On note $H^r(\Omega,O_\Omega)$ (respectivement $H^r(\Omega,\check{O}_\Omega)$) le $r^{i\grave{e}me}$ groupe de cohomologie de $\check{C}$ech des formes diff{\'e}rentielles d{\'e}finies sur $\Omega$ {\`a} valeur dans le faisceau $O_\Omega$ (respectivement $\check{O}_\Omega)$).
\end{notation} 
\section{R{\'e}solution de l'{\'e}quation $du=f$}
\begin{theorem} \label{1}
Soient $M$ une vari{\'e}t{\'e} diff{\'e}rentiable de dimension $n$ et $ \Omega \subset \subset M$ un domaine born{\'e} {\`a} bord lisse de classe $C^\infty$. On suppose que $\Omega$ est contractile et $H^j(b\Omega)$ trivial pour $1 \leq j \leq n-2$. Alors l'{\'e}quation $du = f$ o{\`u} $f$ est une $r$-forme diff{\'e}rentielle de classe $C^\infty$ d{\'e}finie sur $\Omega$ ayant une valeur au bord au sens des courants et $d$-ferm{\'e}e a une solution $u$ qui est une $(r-1)$-forme diff{\'e}rentielle de classe $C^\infty$ admettant une valeur au bord au sens des courants pour $1 \leq r \leq n-2$.
\end{theorem} 
\begin{proof} 
%Puisque $\Omega$ est {\`a} bord lisse alors il existe un domaine $\Omega^{'} \subset \subset M$ avec $\Omega \subset \subset \Omega^{'} \subset \subset M$ tel que $H^j(\Omega^{'}) = H^j(\Omega) = 0$ pour $j\geq 1$ (cf $[7]$).\\
%Si $f$ est une $r$-forme diff{\'e}rentielle {\`a} support compact sur $\bar{\Omega}$, alors $f$ est une $r$-forme diff{\'e}rentielle {\`a} support compact sur $\Omega^{'}$. Puisque $H^r(\Omega^{'}) =0$ par dualit{\'e} de Poincar{\'e} $H^r_c(\Omega^{'}) = 0$ pour $0 \leq r \leq n-1$. Si $df =0$ alors il existe une $(r-1)$-forme diff{\'e}rentielle {\`a} support compact dans $\Omega^{'}$ telle que $du =f$ sur $\Omega^{'}$. Puisque $f\equiv 0$ sur $\Omega^{'} \setminus \Omega$, alors $du = 0$ sur $\Omega^{'} \setminus \Omega$.\\
%Si $r = 1$ alors $du \equiv 0$ sur $\Omega^{'} \setminus \Omega$ et $u$ convient.\\
%Si $r > 1$, comme $H^r(b\Omega) = 0$ pour $r \geq 1$, on a d'apr{\`e}s $[7]$, $H^r(\Omega^{'} \setminus \Omega)$ est isomorphe {\`a} $H^r(b\Omega)$ donc $H^r(\Omega^{'} \setminus \Omega) = 0$ pour $r > 1$ et $u = dv$ o{\`u} $v$ est une $(r-2)$-forme diff{\'e}rentielle sur $\Omega^{'} \setminus \Omega$. Soit $\tilde{V}$une extension de $v$ {\`a} $\Omega$, alors $u -d\tilde{v}$ est une solution de $du = f$ {\`a} support compact sur $\bar{\Omega}$. On a la r{\'e}solution {\`a} support exact. Pour avoir la r{\'e}solution pour les courants prolongeables d{\'e}finis sur $\Omega$, on proc{\'e}de comme dans $[1]$.\\
D'apr\`es $[5]$ si $f$ est une forme diff\'erentielle $d$-ferm\'ee ayant une valeur au bord au sens des courants sur $\Omega$ alors $[f]$ est un courant prolongeable. Puisque $\bar{\Omega}$ est born{\'e}, $[f]$ est d'ordre fini. D'apr\`es $[1]$, $\check{H}^r(\Omega)= 0$, il existe un courant prolongeable $S$ d{\'e}fini sur $\Omega$ tel que $dS = f$. Soit $\tilde{S}$ une extension {\`a} support compact dans $\bar{\Omega}$ de $S$. On a d'apr{\`e}s ($[9]$ page $40$) 
\begin{center}
$\tilde{S} = R\tilde{S} + Ad\tilde{S} + dA\tilde{S}$
\end{center}
$d\tilde{S} =d( R\tilde{S} + Ad\tilde{S})$ et $( R\tilde{S} + Ad\tilde{S})_{\vert \Omega}$ est une autre solution de $du = f$ sur $\Omega$ et est un courant prolongeable. $R\tilde{S}$ est une forme diff{\'e}rentielle de classe $C^\infty$ sur $\Omega$ et $A$ n'augmente pas le support singulier de $d\tilde{S}$. Puisque $d \tilde{S}_{\vert \Omega}= f$ qui est de classe $C^\infty$ alors $Ad\tilde{S}_{\vert \Omega}$ est de classe $C^\infty$. Ainsi $( R\tilde{S} + Ad\tilde{S})_{\vert \Omega}$ est de classe $C^\infty$. Il reste {\`a} montrer que $( R\tilde{S} + Ad\tilde{S})_{\vert \Omega}$ admet une valeur au bord au sens des courants. $R \tilde{S}$ est une forme diff{\'e}rentielle de classe $C^\infty$ donc admet une valeur au bord au sens des courants.\\
Soit $(\varphi_j)_{j \in J}$ une partition de l'unit{\'e} subordonn{\'e}e {\`a} un recouvrement fini $(U_j)_{j \in J}$ de $\bar{\Omega}$ par les ouverts de coordonn{\'e}es locales.\\
On a $Ad \tilde{S} = \displaystyle {\sum_{j \in J} A \varphi_jd \tilde{S} }$ avec $ A \varphi_j d \tilde{S}$ {\`a} support compact sur $U_j$.\\
\begin{center}
$Ad \tilde{S}_{\vert \Omega} = \displaystyle {\sum_{j \in J} (A \varphi_j d \tilde{S})_{\vert \Omega}}$.
\end{center}
Si $U_j \subset \Omega$, alors $ A \varphi_j d \tilde{S}$ est de classe $C^\infty$ {\`a} support compact dans $\Omega$, donc $(A \varphi_j d \tilde{S})_{\vert \Omega}$ admet une valeur au bord au sens des courants.\\
Si $U_j \nsubseteq \Omega$ et $U_j \cap b \Omega \neq \emptyset$; montrons alors que $ A \varphi_j d \tilde{S}$ admet une valeur au bord au sens des courants.\\
Puisque $\varphi_j$ est {\`a} support dans $U_j$ ouvert de coordonn{\'e}es, donc nous sommes ramen{\'e}s {\`a} un domaine born{\'e} de $\mathbb{R}^n$. $ A \varphi_j d \tilde{S}$ est de m{\^e}me nature que l'action du noyau de newton $E(x,y)$ sur $\varphi_j d \tilde{S}$. Puisque $d \tilde{S}$ est d'ordre $l$ {\`a} support compact sur $\bar{\Omega}$ qui prolonge $f$. Donc d'apr{\`e}s $[7]$, $ A \varphi_j d \tilde{S}$ admet une valeur au bord au sens des courants.\\
Puisque $Card(J)< \infty$, donc $\displaystyle {\sum_{j \in J} (A \varphi_j d \tilde{S})_{\vert \Omega}}$ admet une valeur au bord au sens des courants. On a ainsi
\begin{center}
$\tilde{H}^{r}(\Omega) = 0$.
\end{center}
\end{proof}
\section{R{\'e}solution du $\partial \bar{\partial}$ pour les formes diff{\'e}rentielles ayant une valeur au bord au sens des courants}
Comme cons{\'e}quences du th{\'e}or{\`e}me $4.1$ de $[6]$, nous avons le corollaire suivant:
\begin{corollary} \label{4}
Soient $X$ une vari{\'e}t{\'e} analytique complexe de dimension $n$ et $\Omega$ un domaine compl{\'e}tement strictement $(n-1)$-convexe de $X$. Soit $f$ une forme diff{\'e}rentielle de bidegr{\'e} $(0,r)$ d{\'e}finie sur $\Omega$, $\bar{\partial}$-ferm{\'e}e admettant une valeur au bord au sens des courants avec $1 \leq r \leq n-1$. Alors il existe une $(0,r-1)$-forme diff{\'e}rentielle $g$ d{\'e}finie sur $\Omega$ ayant une valeur au bord au sens des courants telle que $\bar{\partial}g = f$.
\end{corollary} 
\begin{proof} 
Consid{\'e}rons la suite suivante:\\
\begin{equation} 
0 \rightarrow \check{O}_\Omega \rightarrow \mathcal{F}^{0,0}(\Omega) \rightarrow \mathcal{F}^{0,1}(\Omega) \rightarrow  \cdots \rightarrow \mathcal{F}^{0,n-1}(\Omega) \rightarrow \mathcal{F}^{0,n}(\Omega) \rightarrow 0.
\end{equation}
On a donc un complexe $(\mathcal{F}^{0,\bullet}(\Omega), \bar{\partial})$ du $\bar{\partial}$ pour les formes diff{\'e}rentielles d{\'e}finies sur $\Omega$ ayant une valeur au bord au sens des courants. Gr{\^a}ce {\`a} la r{\'e}solution locale obtenue dans $[6]$, ce complexe est une r{\'e}solution acyclique du faisceau $\check{O}_\Omega$. Ceci entraine l'isomorphisme $H^r(\Omega,\check{O}_\Omega)\simeq \tilde{H}^{0,r}(\Omega)$. D'apr{\`e}s le corollaire $5.2$ de $[7]$ on a $H^r(\Omega,\check{O}_\Omega)= H^{r}(\Omega, O_\Omega)=0$.\\
Donc
\begin{center}
$\tilde{H}^{0,r}(\Omega) = 0$.
\end{center}
\end{proof}
De fa\c{c}on plus g{\'e}n{\'e}ral si $0 < q < n-1$, on a le th{\'e}or{\`e}me suivant:
\begin{theorem} \label{q}
Soient $X$ une vari{\'e}t{\'e} analytique complexe de dimension $n$ et $\Omega \subset \subset X$ une domaine {\`a} bord lisse de classe $C^\infty$ et compl{\'e}tement strictement $q$-convexe avec $0 \leq q \leq n-1$. Soit $f$ une $(n,r)$-forme diff{\'e}rentielle de classe $C^\infty$, d{\'e}finie sur $\Omega$ et $\bar{\partial}$-ferm{\'e}e admettant une valeur au bord au sens des courants avec $1 \leq n-q \leq r \leq n$. Il existe une $(n , r-1)$-forme diff{\'e}rentielle $g$ de classe $C^\infty$, d{\'e}finie sur $\Omega$ ayant une valeur au bord au sens des courants telle que $\bar{\partial}g = f$.
\end{theorem}
\begin{proof} 
Puisqu'on consid{\`e}re les $(n,r)$-formes diff{\'e}rentielles alors\\ l'op{\'e}rateur $\bar{\partial}$ est {\'e}gal {\`a} l'op{\'e}rateur de diff{\'e}rentiation ext{\'e}rieur $d$.\\
Soit $f \in \tilde{H}^{n,r}(\Omega)$, $[f]$ est un courant prolongeable d{\'e}fini sur $\Omega$. Puisque D'apr{\`e}s $[8]$ on a $\check{H}^{n,r}(\Omega) = 0$ pour $n-q \leq r \leq n$ alors il existe un $(n,r-1)$-courant prolongeable $u$ d{\'e}fini sur $\Omega$ et prolongeable tel que $du = f$. Soit $S$ une extension {\`a} support compact dans $\bar{\Omega}$ de $u$. $S$ est d'ordre fini, $F = dS$  est un courant prolongeable d'ordre fini not{\'e} $m$ et $F_{\vert \Omega} = f$. D'apr{\`e}s ($[9]$ page $40$), on a $S = RS + AdS + dAS$. Or $dS= F$\\ 
$\Rightarrow$ $S= RS+ AF+ dAS$ alors $dS= d(RS+ AF)= F$ donc $(RS + AF)_{\vert \Omega}$ est une autre solution de l'{\'e}quation $du = f$. Or $RS$ est une forme diff{\'e}rentielle de classe $C^\infty$ {\`a} support compact donc admet une valeur au bord au sens des courants. Puisque $A$ n'augmente pas le support singulier, on a si $F$ est de classe $C^\infty$ sur $\Omega$ alors $AF$ est aussi de classe $C^\infty$ sur $\Omega$. Donc la solution $RS + AF$ est de classe $C^\infty$ sur $\Omega$. Il reste {\`a} montrer que $AF$ admet une valeur au bord au sens des courants sur $\Omega$.\\
Soit $(\varphi_j)_{j \in J}$ une partition de l'unit{\'e} subordonn{\'e}e {\`a} un recouvrement fini $(U_j)_{j \in J}$ de $\bar{\Omega}$ par les ouverts de coordonn{\'e}es locales.\\
On a $AF = \displaystyle {\sum_{j \in J} A \varphi_j F}$ avec $ A \varphi_j F$ \`a support compact sur $U_j$.\\
\begin{center}
$AF_{\vert \Omega} = \displaystyle {\sum_{j \in J} (A \varphi_j F)_{\vert \Omega}}$.
\end{center}
Si $U_j \subset \Omega$, alors $ A \varphi_j F$ est de classe $C^\infty$ {\`a} support compact dans $\Omega$, donc $(A \varphi_j F)_{\vert \Omega}$ admet une valeur au bord au sens des courants.\\
Si $U_j \nsubseteq \Omega$ et $U_j \cap b \Omega \neq \emptyset$; montrons alors que $ A \varphi_j F$ admet une valeur au bord au sens des courants.\\
Puisque $\varphi_j$ est {\`a} support dans $U_j$ ouvert de coordonn{\'e}es, donc nous sommes ramen{\'e}s {\`a} un domaine born{\'e} de $\mathbb{C}^n$. $ A \varphi_j F$ est de m{\^e}me nature que l'action du noyau de newton $E(x,y)$ sur $\varphi_j F$. Puisque $F$ est d'ordre $m$ {\`a} support compact sur $\bar{\Omega}$ qui prolonge $f$. Donc d'apr{\`e}s le th{\'e}or{\`e}me \ref{1}, $ A \varphi_j F$ admet une valeur au bord au sens des courants.\\
Puisque $Card(J)< \infty$, donc $\displaystyle {\sum_{j \in J} (A \varphi_j F)_{\vert \Omega}}$ admet une valeur au bord au sens des courants.
\begin{center}
$\tilde{H}^{n,r}(\Omega) = 0$.
\end{center}
\end{proof}
Tenant compte du th{\'e}or{\`e}me \ref{1} et des r{\'e}sultats de r{\'e}solution du $\bar{\partial}$ pour les formes diff{\'e}rentielles ayant une valeur au bord au sens des courants obtenus dans le corollaire \ref{4}, on peut faire la d{\'e}monstration du th{\'e}or{\`e}me \ref{zhou}:
%\begin{theorem} \label{zhou}
%Soient $M$ une vari{\'e}t{\'e} analytique complexe de dimension $n$ et $\Omega \subset \subset M$ un domaine contractile compl{\'e}tement strictement pseudoconvexe {\`a} bord lisse de classe $C^\infty$.Supposons que $H^j(b\Omega)$ trivial pour $j \geq 1$. Alors l'{\'e}quation $\partial \bar{\partial} u = f$, o{\`u} $f$ est une $(p,q)$-forme diff{\'e}rentielle de classe $C^\infty$, $d$-ferm{\'e}e ayant une valeur au bord au sens des courants avec $1 \leq p \leq n$ et $1 \leq q \leq n$ admet une solution $u$ qui est une $(p-1,q-1)$-forme diff{\'e}rentielle de classe $C^\infty$ avec une valeur au bord au sens des courants.
%\end{theorem}
\begin{proof} [D{\'e}monstration th{\'e}or{\`e}me \ref{zhou}] \item
Soit $f$ une $(p,q)$-forme diff{\'e}rentielle de classe $C^\infty$, $d$-ferm{\'e}e ayant une valeur au bord au sens des courants d{\'e}finie sur $\Omega$. Alors d'apr{\`e}s le th{\'e}or{\`e}me \ref{1}, il existe une $(p+q-1)$-forme diff{\'e}rentielle $g$ de classe $C^\infty$ ayant une valeur au bord au sens des courants telle que $dg = f$.\\
On ne perd pas en g{\'e}n{\'e}ralit{\'e} en consid{\'e}rant que $g$ se d{\'e}compose en une $(p-1,q)$-forme diff{\'e}rentielle $g_1$ de classe $C^\infty$ ayant une valeur au bord au sens des courants et en une $(p,q-1)$-forme diff{\'e}rentielle $g_2$ de classe $C^\infty$ ayant une valeur au bord au sens des courants.\\
On a $dg = d(g_1 + g_2) = dg_1 + dg_2 = f$.\\
Comme $d = \partial + \bar{\partial}$, on a pour des raisons de bidegr{\'e} $\partial g_2 = 0$ et $\bar{\partial} g_1 =0$. D'apr{\`e}s le corollaire \ref{4}, on a $g_1 = \bar{\partial}h_1$ et $g_2 = \partial h_2$ avec $h_1$ et $h_2$ des formes diff{\'e}rentielles de classe $C^\infty$ ayant une valeur au bord au sens des courants d{\'e}finies sur $\Omega$.\\
On a $ f = \partial g_1 + \bar{\partial}g_2 = \partial \bar{\partial}h_1 + \bar{\partial} \partial h_2$\\
Or $\partial \bar{\partial} = - \bar{\partial}\partial$
\begin{center}
 $\partial \bar{\partial} h_1 - \partial \bar{\partial} h_2 = \partial \bar{\partial} (h_1-h_2)= f$.
 \end{center}
Posons $u = h_1 -h_2$, $u$ est une $(p-1,q-1)$-forme diff{\'e}rentielle de classe $C^\infty$ ayant une valeur au bord au sens des courants  d{\'e}finie sur $\Omega$ telle que $\partial \bar{\partial} u = f$.
\end{proof}
%%%%%%%%%%%%%%%%%%%%%%%%
Mieux tenant compte du th{\'e}or{\`e}me \ref{q}, on obtient le r{\'e}sultat suivant:
\begin{theorem} \label{sam}
Soient $M$ une vari{\'e}t{\'e} analytique complexe de dimension $n$ et $\Omega \subset \subset M$ un domaine contractile {\`a} bord lisse de classe $C^\infty$. Supposons que $\Omega$ est compl{\'e}tement strictement $q$-convexe avec $0 \leq q \leq n-1$ et $H^j(b\Omega)$ trivial pour $1 \leq j \leq 2n-2$. Alors l'{\'e}quation $\partial \bar{\partial} u = f$, o{\`u} $f$ est une $(n,r)$-forme diff{\'e}rentielle de classe $C^\infty$, $d$-ferm{\'e}e ayant une valeur au bord au sens des courants avec $n-q+1 \leq r \leq n-1$ admet une solution $u$ qui est une $(n-1,r-1)$-forme diff{\'e}rentielle de classe $C^\infty$ avec une valeur au bord au sens des courants.
\end{theorem}
\begin{proof} 
Soit $f$ une $(n,r)$-forme diff{\'e}rentielle de classe $C^\infty$, $d$-ferm{\'e}e ayant une valeur au bord au sens des courants d{\'e}finie sur $\Omega$. Alors d'apr{\`e}s le th{\'e}or{\`e}me \ref{1}, il existe une $(n+r-1)$-forme diff{\'e}rentielle $g$ de classe $C^\infty$ ayant une valeur au bord au sens des courants telle que $dg = f$.\\
$g$ se d{\'e}compose en une $(n-1,r)$-forme diff{\'e}rentielle $g_1$ de classe $C^\infty$ ayant une valeur au bord au sens des courants et en une $(n,r-1)$-forme diff{\'e}rentielle $g_2$ de classe $C^\infty$ ayant une valeur au bord au sens des courants.\\
On a $dg = d(g_1 + g_2) = dg_1 + dg_2 = f$.\\
Comme $d = \partial + \bar{\partial}$, on a pour des raisons de bidegr{\'e} $\partial g_2 = 0$ et $\bar{\partial} g_1 =0$. D'apr{\`e}s le th{\'e}or{\`e}me \ref{q}, on a $g_1 = \bar{\partial}h_1$ et $g_2 = \partial h_2$ avec $h_1$ et $h_2$ des formes diff{\'e}rentielles de classe $C^\infty$ ayant une valeur au bord au sens des courants d{\'e}finies sur $\Omega$.\\
On a $ f = \partial g_1 + \bar{\partial}g_2 = \partial \bar{\partial}h_1 + \bar{\partial} \partial h_2$\\
Or $\partial \bar{\partial} = - \bar{\partial}\partial$
\begin{center}
 $\partial \bar{\partial} h_1 - \partial \bar{\partial} h_2 = \partial \bar{\partial} (h_1-h_2)= f$.
 \end{center}
Posons $u = h_1 -h_2$, $u$ est une $(n-1,r-1)$-forme diff{\'e}rentielle de classe $C^\infty$ ayant une valeur au bord au sens des courants  d{\'e}finie sur $\Omega$ telle que $\partial \bar{\partial} u = f$.
\end{proof}
%%%%%%%%%%%%%%%%%%%%%%
\section{R{\'e}solution du $\partial \bar{\partial}$ pour les formes diff{\'e}rentielles ayant une valeur au bord au sens des courants dans $M \setminus \bar{\Omega}$, o{\`u} $\Omega$ est un domaine contractile compl{\'e}tement strictement pseudoconvexe}
Soit $\Omega$ un domaine d'une vari{\'e}t{\'e} diff{\'e}rentiable $M$ de dimension $n$. Dans cette partie il est question de donner l'analogue du th{\'e}or{\`e}me \ref{zhou} pour $\Omega = M \setminus \bar{D}$, o{\`u} $D \subset \subset M$ est un domaine contractile compl{\'e}tement strictement pseudoconvexe v{\'e}rifiant $H^j(bD)$  trivial  pour $1 \leq j \leq n-2$. Pour cela commen\c{c}ons par donner les r{\'e}sultats suivant:
\begin{corollary} \label{2}
Soit $X$ une vari{\'e}t{\'e} de Stein de dimension $n$. Soit $\Omega \subset X$ tel que $X$ soit une extension $(n-1)$-concave de $\Omega$. Soit $f$ une forme diff{\'e}rentielle de bidegr{\'e} $(0,r)$ sur $\Omega$, $\bar{\partial}$-ferm{\'e}e admettant une valeur au bord au sens des courants, avec $1 \leq r \leq n-2$. Il existe une $(0, r-1)$-forme diff{\'e}rentielle $g$ d{\'e}finie sur $\Omega$ ayant une valeur au bord au sens des courants telle que $\bar{\partial} g = f$.
\end{corollary}
\begin{proof} 
Il s'agit de montrer que $\tilde{H}^{0,r}(\Omega) = 0$ pour $1 \leq r \leq n-2$.\\
Puisque $\Omega$ est concave, on a $\check{O}_\Omega = O_\Omega$. Donc $H^r(\Omega,\check{O}_\Omega)\simeq H^r(\Omega,O_\Omega)$.\\
Consid{\'e}rons la suite suivante:\\
\begin{equation} 
0 \rightarrow \check{O}_\Omega \rightarrow \mathcal{F}^{0,0}(\Omega) \rightarrow \mathcal{F}^{0,1}(\Omega) \rightarrow  \cdots \rightarrow \mathcal{F}^{0,n-1}(\Omega) \rightarrow \bar{\partial}\mathcal{F}^{0,n-1}(\Omega) \rightarrow 0.
\end{equation}
D'apr{\`e}s le th{\'e}or{\`e}me $4.1$ de $[6]$, on sait r{\'e}soudre localement le $\bar{\partial}$ pour les formes diff{\'e}rentielles ayant une valeur au bord au sens des courants sur $\Omega$. Donc la suite est exacte, et puisque le faisceau $\tilde{\mathcal{F}}^{0,r}(\Omega)$ est fin comme faisceau de module sur un faisceau d'anneau de classe $C^\infty$, alors le complexe diff{\'e}rentielle $(\mathcal{F}^{0,\bullet}(\Omega), \bar{\partial})$ des formes diff{\'e}rentielles d{\'e}finies sur $\Omega$ ayant une valeur au bord au sens des courants est une r{\'e}solution acyclique du faisceau $\check{O}_\Omega$. Par cons{\'e}quent pour $0 \leq r \leq n-2$, on a l'isomorphisme fonctoriel suivant:\\

 $H^r(\Omega,\check{O}_\Omega)\simeq \tilde{H}^{0,r}(\Omega):= \frac{ker(\bar{\partial}:\mathbf{E}^{0,r}(\Omega) \rightarrow \mathbf{E}^{0,r+1}(\Omega))}{Im(\bar{\partial}:\mathbf{E}^{0,r-1}(\Omega) \rightarrow \mathbf{E}^{0,r}(\Omega))}$\\
 
 o{\`u} $\mathbf{E}^{0,r}(\Omega):= \Gamma(\Omega,\mathcal{F}^{0,r}(\Omega))$ sont des sections sur $\Omega$ des faisceaux $\mathcal{F}^{0,r}$.\\
 Donc $H^r(\Omega,O_\Omega)\simeq \tilde{H}^{0,r}(\Omega)$. D'apr{\`e}s l'isomorphisme de Dolbeault, on a $H^r(\Omega,O_\Omega)= H^{0,r}(\Omega)$. Puisque $X$ est une extension $(n-1)$-concave de $\Omega$, on a par invariance de la cohomologie: $H^{0,r}(\Omega) = H^{0,r}(X) = 0$ pour $1 \leq r \leq n-1$.\\
 Donc\\
  \begin{center}
  $\tilde{H}^{0,r}(\Omega) = 0$, pour $1 \leq r \leq n-2$.
  \end{center}
\end{proof}
En s'inspirant de la d{\'e}monstration du corollaire \ref{2}, on obtient dans le corollaire suivant, l'analogue global d'un r{\'e}sultat obtenu dans $[6]$.
\begin{theorem} \label{djilo}
Soient $M$ une vari{\'e}t{\'e} analytique complexe de dimension $n$ et $D \subset \subset M$ un domaine contractile compl{\'e}tement strictement pseudoconvexe {\`a} bord lisse tel que $M$ soit une extension $n-1$-convexe de $D$. Posons $\Omega = M \setminus \bar{D}$. Soit $f$ une forme diff{\'e}rentielle de bidegr{\'e} $(0,r)$ d{\'e}finie sur $\Omega$ de classe $C^\infty$, $\bar{\partial}$-ferm{\'e}e avec une valeur au bord au sens des courants, $1 \leq r \leq n-2$. Il existe une $(0,r-1)$-forme diff{\'e}rentielle $g$ d{\'e}finie sur $\Omega$ de classe $C^\infty$, ayant une valeur au bord au sens des courants telle que $\bar{\partial}g = f$. 
\end{theorem}
\begin{proof}
Il s'agit de montrer que $\tilde{H}^{0,r}(\Omega) = 0$ pour $1 \leq r \leq n-2$.\\
Puisque $\Omega$ est concave, on a $\check{O}_\Omega = O_\Omega$. Donc $H^r(\Omega,\check{O}_\Omega)\simeq H^r(\Omega,O_\Omega)$.\\
Consid{\'e}rons la suite suivante:\\
\begin{equation} 
0 \rightarrow \check{O}_\Omega \rightarrow \mathcal{F}^{0,0}(\Omega) \rightarrow \mathcal{F}^{0,1}(\Omega) \rightarrow  \cdots \rightarrow \mathcal{F}^{0,n-1}(\Omega) \rightarrow \mathcal{F}^{0,n}(\Omega) \rightarrow 0.
\end{equation}
D'apr{\`e}s $[6]$, on sait r{\'e}soudre localement le $\bar{\partial}$ pour les formes diff{\'e}rentielles ayant une valeur au bord au sens des courants sur $\Omega$. Donc la suite est exacte, et puisque le faisceau $\tilde{\mathcal{F}}^{0,r}(\Omega)$ est fin comme faisceau de module sur un faisceau d'anneau de classe $C^\infty$, alors le complexe diff{\'e}rentielle $(\mathcal{F}^{0,\bullet}(\Omega), \bar{\partial})$ des formes diff{\'e}rentielles d{\'e}finies sur $\Omega$ ayant une valeur au bord au sens des courants est une r{\'e}solution acyclique du faisceau $\check{O}_\Omega$. Par cons{\'e}quent pour $0 \leq r \leq n-2$, on a l'isomorphisme fonctoriel suivant:\\

 $H^r(\Omega,\check{O}_\Omega)\simeq \tilde{H}^{0,r}(\Omega):= \frac{ker(\bar{\partial}:\mathbf{E}^{0,r}(\Omega) \rightarrow \mathbf{E}^{0,r+1}(\Omega))}{Im(\bar{\partial}:\mathbf{E}^{0,r-1}(\Omega) \rightarrow \mathbf{E}^{0,r}(\Omega))}$\\
 
 o{\`u} $\mathbf{E}^{0,r}(\Omega):= \Gamma(\Omega,\mathcal{F}^{0,r}(\Omega))$ sont des sections sur $\Omega$ des faisceaux $\mathcal{F}^{0,r}$.\\
 Donc $H^r(\Omega,O_\Omega)\simeq \tilde{H}^{0,r}(\Omega)$. D'apr{\`e}s l'isomorphisme de Dolbeault, on a $H^r(\Omega,O_\Omega)= H^{0,r}(\Omega)$. Puisque $M$ est une extension $n-1$-concave de $\Omega$, on a par invariance de la cohomologie: $H^{0,r}(\Omega) = H^{0,r}(M)$ pour $0 \leq r \leq n-2$.\\
 D'apr{\`e}s $[2]$ $H^{0,r}(D) \simeq H^{0,r}(M) = 0$ pour $1 \leq r \leq n$.\\
 Donc\\
  \begin{center}
  $\tilde{H}^{0,r}(\Omega) = 0$, pour $1 \leq r \leq n-2$.
  \end{center}
%De plus si $M$ est une vari{\'e}t{\'e} de Stein, on a   $ H^{0,r}(M) = 0$ $\Rightarrow$  $H^{0,r}(\Omega) = 0$.\\
%Donc\\
%\begin{center}
 % $\tilde{H}^{0,r}(\Omega) = 0$, pour $1\leq r \leq q-1$.
  %\end{center}
\end{proof}
De fa\c{c}on plus g{\'e}n{\'e}ral si $0 < q < n-1$, on a le th{\'e}or{\`e}me suivant:
\begin{theorem} \label{djilo}
Soient $M$ une vari{\'e}t{\'e} analytique complexe de dimension $n$ et $D \subset \subset M$ un domaine contractile compl{\'e}tement strictement pseudoconvexe {\`a} bord lisse tel que $M$ soit une extension $q$-convexe de $D$ avec $q \geq \frac{n+1}{2}$. Posons $\Omega = M \setminus \bar{D}$. Soit $f$ une forme diff{\'e}rentielle de bidegr{\'e} $(0,r)$ d{\'e}finie sur $\Omega$ de classe $C^\infty$, $\bar{\partial}$-ferm{\'e}e avec une valeur au bord au sens des courants, $1 \leq r \leq q-1$. Il existe une $(0,r-1)$-forme diff{\'e}rentielle $g$ d{\'e}finie sur $\Omega$ de classe $C^\infty$, ayant une valeur au bord au sens des courants telle que $\bar{\partial}g = f$. 
\end{theorem}
\begin{proof}
Il s'agit de montrer que $\tilde{H}^{0,r}(\Omega) = 0$ pour $1 \leq r \leq q-1$\\
Puisque $\Omega$ est concave, on a $\check{O}_\Omega = O_\Omega$. Donc $H^r(\Omega,\check{O}_\Omega)\simeq H^r(\Omega,O_\Omega)$.\\
Consid{\'e}rons la suite suivante:\\
\begin{equation} 
0 \rightarrow \check{O}_\Omega \rightarrow \mathcal{F}^{0,0}(\Omega) \rightarrow \mathcal{F}^{0,1}(\Omega) \rightarrow  \cdots \rightarrow \mathcal{F}^{0,n-1}(\Omega) \rightarrow \mathcal{F}^{0,n}(\Omega) \rightarrow 0
\end{equation}
D'apr{\`e}s $[6]$, on sait r{\'e}soudre localement le $\bar{\partial}$ pour les formes diff{\'e}rentielles ayant une valeur au bord au sens des courants sur $\Omega$. Donc la suite est exacte, et puisque le faisceau $\tilde{\mathcal{F}}^{0,r}(\Omega)$ est fin comme faisceau de module sur un faisceau d'anneau de classe $C^\infty$, alors le complexe diff{\'e}rentielle $(\mathcal{F}^{0,\bullet}(\Omega), \bar{\partial})$ des formes diff{\'e}rentielles d{\'e}finies sur $\Omega$ ayant une valeur au bord au sens des courants est une r{\'e}solution acyclique du faisceau $\check{O}_\Omega$. Par cons{\'e}quent pour $0 \leq r \leq q-1$, on a l'isomorphisme fonctoriel suivant:\\

 $H^r(\Omega,\check{O}_\Omega)\simeq \tilde{H}^{0,r}(\Omega):= \frac{ker(\bar{\partial}:\mathbf{E}^{0,r}(\Omega) \rightarrow \mathbf{E}^{0,r+1}(\Omega))}{Im(\bar{\partial}:\mathbf{E}^{0,r-1}(\Omega) \rightarrow \mathbf{E}^{0,r}(\Omega))}$\\
 
 o{\`u} $\mathbf{E}^{0,r}(\Omega):= \Gamma(\Omega,\mathcal{F}^{0,r}(\Omega))$ sont des sections sur $\Omega$ des faisceaux $\mathcal{F}^{0,r}$.\\
 Donc $H^r(\Omega,O_\Omega)\simeq \tilde{H}^{0,r}(\Omega)$. D'apr{\`e}s l'isomorphisme de Dolbeault, on a $H^r(\Omega,O_\Omega)= H^{0,r}(\Omega)$. Puisque $M$ est une extension $q$-concave de $\Omega$, on a par invariance de la cohomologie: $H^{0,r}(\Omega) = H^{0,r}(M)$ pour $0 \leq r \leq q-1$.\\
 D'apr{\`e}s $[3]$ $H^{0,r}(D) \simeq H^{0,r}(M) = 0$ pour $n-q \leq r \leq n$.\\
 Donc\\
  \begin{center}
  $\tilde{H}^{0,r}(\Omega) = 0$, pour $n-q \leq r \leq q-1$.
  \end{center}
De plus si $M$ est une vari{\'e}t{\'e} de Stein, on a   $ H^{0,r}(M) = 0$ $\Rightarrow$  $H^{0,r}(\Omega) = 0$.\\
Donc\\
\begin{center}
  $\tilde{H}^{0,r}(\Omega) = 0$, pour $1\leq r \leq q-1$.
  \end{center}
\end{proof}
%Tenant compte du th{\'e}or{\`e}me \ref{djilo}, nous avons le th{\'e}or{\`e}me suivant:
%%%%%%%%%%%%%%%%%%%%%%%%%%
%\begin{theorem} \label{aff}
%Soient $M$ une vari{\'e}t{\'e} complexe de dimension $n$, $D \subset \subset M$ un domaine contractile compl{\'e}tement strictement pseudoconvexe {\`a} bord lisse de classe $C^\infty$ tel que $M$ soit une extension $q$-convexe de $D$ avec $q \geq \frac{n+1}{2}$. Supposons que $H^j(bD)$  trivial  pour $1 \leq j \leq n-2$. Posons $\Omega = M \setminus \bar{D}$. Si $f$ est une $(p,r)$-forme diff{\'e}rentielle de classe $C^\infty$, $d$-ferm{\'e}e et admettant une valeur au bord au sens des courants sur $\Omega$ avec $1 \leq p \leq n$ et $n-q \leq r \leq q-1$, alors il existe une $(p-1,r-1)$-forme diff{\'e}rentielle $u$ de classe $C^\infty$ d{\'e}finie sur $\Omega$ ayant une valeur au bord au sens des courants telle que $\partial \bar{\partial} u = f$.
%\end{theorem}
Pour faire la preuve du th{\'e}or{\`e}me \ref{aff}, nous avons besoin du lemme suivant:
\begin{lemma} \label{aff1}
Soient $M$ une vari{\'e}t{\'e} diff{\'e}rentiable de dimension $n$, $D \subset \subset M$ un domaine contractile {\`a} bord lisse de classe $C^\infty$ tel que $M$ soit une extension contractile de $D$ avec $H^j(bD)$  trivial  pour $1 \leq j \leq n-2$. Posons $\Omega = M \setminus \bar{D}$. Si $\stackrel{\circ}{\bar{\Omega}} = \Omega$ et $f$ est une $r$-forme diff{\'e}rentielle de classe $C^\infty$, $d$-ferm{\'e}e admettant une valeur au bord au sens des courants sur $\Omega$ avec $0 \leq r \leq n$, alors il existe une $(r-1)$-forme diff{\'e}rentielle $g$ de classe $C^\infty$ sur $\Omega$ ayant une valeur au bord au sens des courants telle que $dg = f$. 
\end{lemma}
\begin{proof}
Soit $f$ une forme diff{\'e}rentielle de classe $C^\infty$, $d$-ferm{\'e}e ayant une valeur au bord au sens des courants sur $\Omega$, d'apr{\`e}s le lemme $4.1$ de $[6]$, $[f]$ est un courant prolongeable d'ordre fini. D'apr{\`e}s $[2]$, $\check{H}^r(\Omega) = 0$.\\ 
Nous avons besoin du lemme suivant pour la suite de la d{\'e}monstration.
%\begin{lemma} \label{lem}
%Soit $T \in \check{D}_{l}^{' (0,r)} (\Omega) \cap ker \bar{\partial}$, alors il existe $S \in \check{D}_{l}^{' (0,r-1)} (\Omega)$ tel que $\bar{\partial} S = T$, $1\leq r \leq n-2$.
%\end{lemma}
%\begin{proof}[D\'emonstration du Lemme \ref{lem}]
%Elle est identique \`a la d\'emonstration du Lemme $4.1$ de $[2]$. Il suffit de remplacer les formes diff\'erentielles de classe $C^\infty$ \`a support compact par les formes diff\'erentielles de classe $C^l$ \`a support compact.
%\end{proof}
\begin{lemma} \label{c}
L'application naturelle
\begin{center}
$\check{H}_k^r(\Omega) \rightarrow \check{H}^r(\Omega)$
\end{center}
est un isomorphisme pour $0 \leq r \leq n$.
\end{lemma}
\begin{proof} [D{\'e}monstration du lemme \ref{c}] \item
\begin{enumerate}
\item[(a)] Injectivit{\'e}: Soit $[T] \in \check{H}_k^r(\Omega)$ tel que $[T] = 0$ dans $\check{H}^r(\Omega)$, il existe $S \in \check{D}^{r-1}(\Omega)$ tel que $dS = T$. Soit $\check{S}$ un prolongement de $S$ {\`a} $M$, consid{\'e}rons le courant $\check{T}= d\check{S}$ qui est un prolongement de $T$ {\`a} $M$. D'apr{\`e}s ($[9]$ page $40$), 
\begin{center}
$\check{S}=R\check{S} +Ad\check{S} + dA\check{S}$\\
$d\check{S} = d(R\check{S} +Ad\check{S})$.
\end{center}
Puisque $A$ n'augmente pas le support singulier donc $Ad\check{S}_{\vert \Omega}$ est d'ordre $k$. $R\check{S}$ est une forme diff{\'e}rentielle de classe $C^\infty$. Ainsi $(R\check{S} +Ad\check{S})_{\vert \Omega}$ est d'ordre $k$ et $d(R\check{S} +Ad\check{S})_{\vert \Omega}= T$. Donc $[T]= 0$ dans $\check{H}_k^r(\Omega)$, d'o{\`u} l'injectivit{\'e} de l'application naturelle.
\item[(b)] Surjectivit{\'e}: Soit $[T] \in  \check{H}^r(\Omega)$ $\Rightarrow$ $dT = 0$, on a
\begin{center}
$T=RT + dAT$.\\
\end{center}
%\begin{center}
%$T = (R\check{T} +Ad\check{T})_{\vert \Omega} +dA\check{T}_{\vert \Omega} $.
%\end{center}
On a $dT=d(RT)= 0$, donc $[RT]\in  \check{H}_k^r(\Omega)$ et $[T] = [RT]$, d'o{\`u} la surjectivit{\'e}.
\end{enumerate}
\end{proof}
Donc  il existe un $(r-1)$-courant prolongeable $u$ d{\'e}fini sur $\Omega$ tel que $du=f$. Soit $S$ un prolongement de $u$ sur $M$ avec $dS = F$ o{\`u} $F$ est un prolongement du courant $[f]$ et donc d'ordre fini $m$.  On a $S=RS +AdS + dAS$.\\
Posons $\check{S} = RS +AdS$, $RS$ {\'e}tant une forme diff{\'e}rentielle de classe $C^\infty$ sur $M$ donc $RS_{\vert \Omega}$ admet une valeur au bord au sens des courants. Puisque $A$ n'augmente pas le support singulier donc $AdS_{\vert \Omega}$ est de classe $C^\infty$ et admet comme dans la d{\'e}monstration du th{\'e}or{\`e}me \ref{1} une valeur au bord au sens des courants.\\
$\check{S}_{\vert \Omega}$ admet aussi une valeur au bord au sens des courants et $d \check{S}_{\vert \Omega} = f$.
\end{proof}
\begin{proof} [D{\'e}monstration du Th{\'e}or{\`e}me \ref{aff}] \item 
Soit $f$ une $(p,q)$-forme diff{\'e}rentielle de classe $C^\infty$, $d$-ferm{\'e}e ayant une valeur au bord au sens des courants sur $\Omega$. D'apr{\`e}s le lemme \ref{aff1}, il existe une $(p+q-1)$-forme diff{\'e}rentielle $h$ de classe $C^\infty$ ayant une valeur au bord au sens des courants telle que $dh = f$. On a $h= h_1+ h_2$ o{\`u} $h_1$ et $h_2$ sont respectivement une $(p-1,q)$-forme diff{\'e}rentielle de classe $C^\infty$ ayant une valeur au bord au sens des courants et une $(p,q-1)$-forme diff{\'e}rentielle de classe $C^\infty$ ayant une valeur au bord au sens des courants. On a $dh= dh_1 +dh_2=f$.\\
Comme $d= \partial + \bar{\partial}$ et pour des raisons de bidegr{\'e} on a $\partial h_2 =0$ et $\bar{\partial} h_1 =0$. D'apr{\`e}s le th{\'e}or{\`e}me \ref{djilo}, $h_1 =\bar{\partial}g_1$ et $h_2= \partial g_2$ avec $g_1$ et $g_2$ des formes diff{\'e}rentielles de classe $C^\infty$ ayant une valeur au bord au sens des courants d{\'e}finies sur $\Omega$. On a
\begin{center}
$f = \partial h_1+ \bar{\partial}h_2= \partial \bar{\partial}g_1 +\bar{\partial}\partial g_2 = \partial \bar{\partial}(g_1 - g_2)$.
\end{center}
Posons $u= g_1 - g_2$, $u$ est une $(p-1,q-1)$-forme diff{\'e}rentielle de classe $C^\infty$ ayant une valeur au bord au sens des courants d{\'e}finie sur $\Omega$ telle que $\partial \bar{\partial} u = f$.\\
%Mieux si $M$ est une vari{\'e}t{\'e} de Stein alors le th{\'e}or{\`e}me \ref{aff} reste vrai pour $1 \leq p \leq n$ et $1 \leq r \leq q-1$
\end{proof}
       
\end{document}